\documentclass[12pt,draft]{article}
 \usepackage{amsfonts,amsmath,amssymb,amsthm}

 \newcommand{\Nr}{N\makebox[0pt][l]{\raisebox{5pt}{$\scriptscriptstyle{{\rm o}}$}}-}

 \newcommand{\f}{\varphi}

 \newcommand\diag{\mathop{diag}}

 \def\dC{{\mathbb C}}
 
 \def\dZ{{\mathbb Z}}

 \newcommand{\lam}{\lambda}
 \newcommand{\vek}{\overrightarrow}

 \newtheorem{theorem}{Theorem}
 \newtheorem{lemma}[theorem]{Lemma}

 \newtheorem{definition}[theorem]{Definition}
 \newtheorem{corollary}[theorem]{Corollary}
 \newtheorem{proposition}[theorem]{Proposition}

 \begin{author}
 {S.~Hassi, L.L.~Oridoroga}
 \end{author}

 \begin{title}
 {Theorem of completeness for a Dirac-type operator with
 generalized $\lambda$-depending boundary conditions}
 \end{title}

 \date{}

\begin{document}
 \maketitle


 \begin{abstract}
 A completeness theorem is proved involving a system of
 integro-differential equations with some $\lambda$-depending boundary
 conditions. Also some sufficient conditions for the root functions
 to form a Riesz basis are established.
 \end{abstract}

 1. It is well known \cite{Mar} that the system of eigenfunctions and
 associate functions (SEAF) of the Sturm -- Liouville problem
 \begin{equation} \label {1}
 -y''+q(x)y=\lam^2 y,
 \end{equation}
 \begin{equation} \label{2}
 y'(0)-h_0y(0)=y'(1)-h_1y(1)=0,
 \end{equation}
 is complete in $L_2[0,1]$ for arbitrary complex valued potential
 $q\in L_1[0,1]$ and $h_0,h_1\in \dC$. A similar result is also known
 for arbitrary nondegenerate boundary conditions (see \cite{Mar}).

 A completeness result for a boundary value problem
 of arbitrary order differential equations of the form
 \begin{equation}
 y^{(n)}+\sum\limits_{j=0}^{n-2}q_j(x)y=\lam^n y,
 \end{equation}
 with separated boundary conditions, has been announced by
 M.V. Keldysh \cite{Kel} and was first proved by A.A. Shkalikov
 \cite{Shkal}.

 In \cite{Mal+Or} M.M. Malamud and one of the authors have
 generalized the above mentioned results from \cite{Mar} to the
 case of first order systems with arbitrary boundary conditions
 (not depending on a spectral parameter).

 In \cite{Tar1} and \cite{Tar2} the completeness results for
 the problem \eqref{1}, \eqref{2} have been generalized to
 the case of nonlinear $\lambda$-depending boundary conditions
 of the form
 \begin{equation}  \label {a1}
 \begin{cases}
 P_{11}(\lam)y(0)+P_{12}(\lam)y'(0)=0 \\
 P_{21}(\lam)y^2(\frac 12)+P_{22}(\lam)y(\tfrac  12)y'(\tfrac  12)+
 P_{23}(\lam){y'}^2(\frac 12)=0
 \end{cases}
 \end{equation}
 and of the form
 \begin{equation}\label {a2}
 \begin{cases}
 \begin{split}
 & P_{10}(\lam)y^2(0)+P_{11}(\lam){y'}^2(0)+
      P_{12}(\lam)y^2(\tfrac 13)+P_{13}(\lam){y'}^2(\tfrac 13) \\
 &+P_{14}(\lam)y(0)y'(0)+P_{15}(\lam)y(0)y(\tfrac 13)
  +P_{16}(\lam)y(0)y'(\tfrac 13) \\
 &+P_{17}(\lam)y'(0)y(\tfrac 13)+P_{18}(\lam)y'(0)y'(\tfrac 13)
  +P_{19}(\lam)y(\tfrac 13)y'(\tfrac 13)=0 \\
 \end{split} \\ \\
 \begin{split}
 & P_{20}(\lam)y^2(0)+P_{21}(\lam){y'}^2(0)+
      P_{22}(\lam)y^2(\tfrac 13)+P_{23}(\lam){y'}^2(\tfrac 13) \\
 & +P_{24}(\lam)y(0)y'(0)+P_{25}(\lam)y(0)y(\tfrac 13)+
      P_{26}(\lam)y(0)y'(\tfrac 13) \\
 & +P_{27}(\lam)y'(0)y(\tfrac 13)+P_{28}(\lam)y'(0)y'(\tfrac 13)+
      P_{29}(\lam)y(\tfrac 13)y'(\tfrac 13)=0,
 \end{split}
 \end{cases}
 \end{equation}
 where $P_{ij}(\lam)$ are polynomials.

 Moreover, in \cite{OrMFAT} analogous results were
 obtained for a system with a pair of separated $\lambda$-depending
 boundary conditions similar to the conditions \eqref{a1} and
 \eqref{a2}.

 In the recent papers \cite{TroYa}, \cite{TroYa2}
 a related problem concerning the Riesz basis property of the SEAF for a first-order system
 of the form \eqref{3} given below with separated
 boundary conditions, not depending on a spectral parameter,
 has been established. In the present paper completeness and Riesz basis property
 of the SEAF are considered for Dirac-type systems with certain
 $\lambda$-depending boundary conditions.
 Naturally the result cover the case of Dirac operators, which have been more extensively studied
 in the literature. In particular, we wish to mention the recent studies on spectral decompositions
 of 1D periodic Dirac operators and related convergence results by B. Mityagin and P.
 Djakov; see \cite{Mit} and \cite{DMit1}, \cite[Section~4]{DMit2}.

 The paper is organized as follows.
 In Section 1 we prove a completeness result for
 the  first order systems of certain integro-differential equations
 involving general linear or quadratic
 $\lambda$-depending boundary conditions.
 More precisely, let $B=\diag(a^{-1},b^{-1})$ be a $2\times 2$
 diagonal matrix with $a<0<b$.
 Consider in $L_2[0,1]\oplus L_2[0,1]$ a boundary value problem for
 the first order system of ordinary integro-differential equations of the form
 \begin{equation} \label {3}
 \frac 1iBy'+Q(x)y+\int_0^xM(x,t)y(t)\,dt=\lam y.
 \end{equation}
 Here
 \begin{equation*}
 Q(x,t)=\begin{pmatrix} 0&q_1(x,t)\\q_2(x,t)&0\end{pmatrix},\quad
 M(x,t)=\begin{pmatrix} M_{11}(x,t)&M_{12}(x,t)\\M_{21}(x,t)&M_{22}(x,t)\end{pmatrix},\quad
 y(x)=\binom{y_1(x)}{y_2(x)},
 \end{equation*}
 where it is assumed that $q_j\in L_1[0,1]$ and
 $M_{ij}\in L_\infty(\Omega)$,
 $\Omega=\{0\le t\le x\le 1\}$, $i,j=1,2$.
 Two types of $\lambda$-depending boundary conditions will be treated.
 Namely:

 (i) arbitrary linear conditions of the form
 \begin{equation} \label{5}
 \begin{cases}
 P_{11}(\lam)y_1(0)+P_{12}(\lam)y_2(0)+P_{13}(\lam)y_1(1)+P_{14}(\lam)y_2(1)=0 \\
 P_{21}(\lam)y_1(0)+P_{22}(\lam)y_2(0)+P_{23}(\lam)y_1(1)+P_{24}(\lam)y_2(1)=0
 \end{cases}
 \end{equation}

 and:

 (ii) arbitrary quadratic conditions of the form
 \begin{equation} \label{6}
 \begin{cases}
 \begin{split}
 & P_{10}(\lam)y_1^2(0)+P_{11}(\lam)y_2^2(0)+
      P_{12}(\lam)y_1^2(\tfrac  12)+P_{13}(\lam)y_2^2(\tfrac  12) \\
 &+P_{14}(\lam)y_1(0)y_2(0)+P_{15}(\lam)y_1(0)y_1(\tfrac  12)+
      P_{16}(\lam)y_1(0)y_2(\tfrac  12) \\
 &+P_{17}(\lam)y_2(0)y_1(\tfrac  12)+P_{18}(\lam)y_2(0)y_2(\tfrac  12)+
      P_{19}(\lam)y_1(\tfrac  12)y_2(\tfrac  12)=0 \\
 \end{split} \\ \\
 \begin{split}
 &P_{20}(\lam)y_1^2(0)+P_{21}(\lam)y_2^2(0)+
      P_{22}(\lam)y_1^2(\tfrac  12)+P_{23}(\lam)y_2^2(\tfrac  12) \\
 &+P_{24}(\lam)y_1(0)y_2(0)+P_{25}(\lam)y_1(0)y_1(\tfrac  12)+
      P_{26}(\lam)y_1(0)y_2(\tfrac  12) \\
 &+P_{27}(\lam)y_2(0)y_1(\tfrac  12)+P_{28}(\lam)y_2(0)y_2(\tfrac  12)+
      P_{29}(\lam)y_1(\tfrac  12)y_2(\tfrac  12)=0,
 \end{split}
 \end{cases}
 \end{equation}
 where $P_{ij}(\lam)$ are polynomials.

 In Section 2 some general sufficient conditions for polynomials $P_{ij}$ are established in order that
 the SEAF of the problem \eqref{3} with separated $\lambda$-depending
 boundary conditions forms a Riesz basis.

 Some of the main results of this paper has been announced without proofs in \cite{HaOr}, \cite{HaOr2}.

 \section{Theorems on completeness of SEAF}

 In this section some sufficient conditions for the completeness of the SEAF
 of the problems \eqref{3}, \eqref{5} and \eqref{3}, \eqref{6}
 in $L^2[0,1]\oplus L^2[0,1]$ are established.
 The starting point is to estimate the growth of the solution
 of the Cauchy problem for the system \eqref{3} with
 special initial conditions.

 Let
 \begin{equation} \label{7}
 \vek\f_\alpha(x;\lambda)=\binom{\f_{\alpha 1}(x;\lambda)}{\f_{\alpha 2}(x;\lambda)} \quad \text{ and }
 \quad \vek\psi_\alpha(x;\lambda)=\binom{\psi_{\alpha 1}(x;\lambda)}{\psi_{\alpha 2}(x;\lambda)}
 \end{equation}
 be the solutions of the Cauchy problem for the system \eqref{3} with
 the initial conditions
 \begin{equation} \label{8}
 {\f_{\alpha 1}(\alpha;\lambda)}={\psi_{\alpha 2}(\alpha;\lambda)}=1 \quad \text{ and } \quad
 {\f_{\alpha 2}(\alpha;\lambda)}={\psi_{\alpha 1}(\alpha;\lambda)}=0,
 \end{equation}
 where $\alpha\in[0,1].$
 The next lemma gives some estimates for the growth of
 $\f_{0j}(x;\lambda)$ and $\psi_{0j}(x;\lambda)$, $j=1,2$.

 \begin{lemma}
 \label{lem.1}
 The functions $\f_{0j}(x;\lambda)$ and $\psi_{0j}(x;\lambda)$, $j=0,1$, satisfy the estimates (as $\lambda\to\infty$)
 \begin{equation}\label{11}
 \begin{split}
 & \f_{01}(x;\lambda)=(1+O(\tfrac1{\Im\lambda}))\exp(a\lambda ix),\quad
   \f_{02}(x;\lambda)=O(\tfrac1{\Im\lambda})\exp(a\lambda ix); \\
 &\psi_{01}(x;\lambda)=O(\tfrac1{\Im\lambda})\exp(a\lambda ix), \quad
  \psi_{02}(x;\lambda)=O(\tfrac1{\Im\lambda})\exp(a\lambda ix);
 \end{split}
 \end{equation}
 when $\lambda\in\dC^+$, and the estimates
 \begin{equation}\label{12}
 \begin{split}
 &\f_{01}(x;\lambda)=O(\tfrac1{\Im\lambda})\exp(b\lambda ix), \quad
  \f_{02}(x;\lambda)=O(\tfrac1{\Im\lambda})\exp(b\lambda ix);  \\
 &\psi_{01}(x;\lambda)=O(\tfrac1{\Im\lambda})\exp(b\lambda ix), \quad
  \psi_{02}(x;\lambda)=(1+O(\tfrac1{\Im\lambda}))\exp(b\lambda ix);
 \end{split}
 \end{equation}
 when $\lambda\in\dC^-$.
 \end{lemma}

 \begin{proof}
We prove the first of the estimates in \eqref{11}. All the other
estimates can be proved similarly.

According to \cite{Mal} the system \eqref{3} admits a triangular
transformation operator. This means that the solution $\varphi_0$
admits a representation
        \begin{equation} \label{l1.1}
 \varphi_0(x;\lambda)=\binom{e^{ia\lambda x}}{0}+
 \int_0^xK(x,t)\binom{e^{ia\lambda t}}{e^{ib\lambda t}}\,dt,
 \end{equation}
 where $K(x,t):=
 \begin{pmatrix}
 K_{11}(x,t) & K_{12}(x,t) \\
 K_{21}(x,t) & K_{22}(x,t)
 \end{pmatrix}
 \in L_\infty(\Omega)\otimes \mathbb{C}^{2\times 2}$.
In particular,
 \begin{equation*}
 \varphi_{01}(x;\lambda)=e^{ia\lambda x}+
 \int_0^xK_{11}(x,t)e^{ia\lambda t}\,dt+
 \int_0^xK_{12}(x,t)e^{ib\lambda t}\,dt.
 \end{equation*}

If $\Im\lambda >0$, then $|e^{ib\lambda t}|<1$ and
$\int_0^xK_{12}(x,t)e^{ib\lambda t}\,dt=O(1)$, since $a<0<b$ and
$K_{12}(x,t)$ is bounded, and then, in particular,
$\int_0^xK_{12}(x,t)e^{ib\lambda
t}\,dt=O(\tfrac1{\Im\lambda})\exp(a\lambda
 ix)$.

 Moreover, since $K_{11}(x,t)$ is bounded, one has
 \begin{equation*}
 \int_0^xK_{11}(x,t)e^{ia\lambda t}\,dt=O\left(\int_0^x |e^{ia\lambda t}|\,dt\right)=
 O\left(\int_0^x e^{-a(\Im\lambda) t}\,dt\right)=
 O\left(\left|\frac{1}{\Im\lambda}\right|\right)e^{ia\lambda x}.
 \end{equation*}
Therefore, for  $\Im\lambda >0$ one gets
 \begin{equation*}
 \varphi_{01}(x;\lambda)=e^{ia\lambda x}+
 O\left(\left|\frac{e^{ia\lambda x}}{\Im\lambda}\right|\right)=
 \left(1+O\left(\frac1{\Im\lambda }\right)\right)e^{ia\lambda x}.
 \end{equation*}

 \end{proof}

 If the function $Q(x)$ is differentiable, the above
 estimates can be strengthened as follows.

 \begin{lemma} \label {lem.2}
 Let $Q(x)$ be differentiable. Then the functions
  $\vek\f_{j}(x;\lambda)$ and $\vek\psi_{j}(x;\lambda)$, $j \in \{ 0,1 \}$
 satisfy the estimates
 \begin{equation} \label{l2.1}
 \begin{split}
 \f_{01}(x;\lambda)&=\exp(a\lambda ix)+\tfrac1{\lambda} O(\exp(a\lambda ix)), \quad
 \f_{02}(x;\lambda)=\tfrac1{\lambda} O(\exp(a\lambda ix));
  \\
 \psi_{01}(x;\lambda)&=\tfrac1{\lambda} O(\exp(a\lambda ix)), \quad
 \psi_{02}(x;\lambda)=\tfrac1{\lambda} O(\exp(a\lambda ix));
  \\  \\
 \f_{11}(x;\lambda)&=\tfrac1{\lambda} O(\exp(b\lambda i(x-1))), \quad
 \f_{12}(x;\lambda)=\tfrac1{\lambda} O(\exp(b\lambda i(x-1)));
  \\
 \psi_{11}(x;\lambda)&=\tfrac1{\lambda} O(\exp(b\lambda i(x-1))), \quad
 \psi_{12}(x;\lambda)=\exp(b\lambda i(x-1))+
                  \tfrac1{\lambda} O(\exp(b\lambda i(x-1))),
 \end{split}
 \end{equation}
 when $\lambda\in\dC^+$, and the estimates
 \begin{equation} \label{l2.2}
 \begin{split}
 \f_{01}(x;\lambda)&=\tfrac1{\lambda} O(\exp(b\lambda ix)), \quad
 \f_{02}(x;\lambda)=\tfrac1{\lambda} O(\exp(b\lambda ix));
  \\
 \psi_{01}(x;\lambda)&=\tfrac1{\lambda} O(\exp(b\lambda ix)), \quad
 \psi_{02}(x;\lambda)=\exp(b\lambda ix)+\tfrac1{\lambda} O(\exp(b\lambda ix));
  \\  \\
 \f_{11}(x;\lambda)&=\exp(a\lambda i(x-1))+
               \tfrac1{\lambda} O(\exp(a\lambda i(x-1))), \quad
 \f_{12}(x;\lambda)=\tfrac1{\lambda} O(\exp(a\lambda i(x-1)));
  \\
 \psi_{11}(x;\lambda)&=\tfrac1{\lambda} O(\exp(a\lambda i(x-1))), \quad
 \psi_{12}(x;\lambda)=\tfrac1{\lambda} O(\exp(a\lambda i(x-1))),
 \end{split}
 \end{equation} when $\lambda\in\dC^-$.
 \end{lemma}

 \begin{proof}
 As in Lemma \ref{lem.1} we just prove the first of the estimates in \eqref{l2.1}.
 For this purpose the formula \eqref{l1.1} in the proof of Lemma \ref{lem.1} will be used.
 In \cite{Mal} it was show that if $q\in C^1[0,1]$ then $K_{ij}(x,t)\in C^1(\Omega),$ too.
 Now integration by parts yields

 \begin{multline*}
 \int_0^xK_{11}(x,t)e^{ia\lambda t}\,dt=
 \frac 1{ia\lambda} \left(\left.K_{11}(x,t)e^{ia\lambda t}\right|_{t=0}^x-
 \int_0^x \frac{\partial }{\partial t}K_{11}(x,t)e^{ia\lambda t}\,dt\right) \\
=\frac{K_{11}(x,x)e^{ia\lambda x}-K_{11}(x,0)}{ia\lambda}-
 \frac 1{ia\lambda} O\left(\left|\frac{e^{-a(\Im\lambda) x}-1}{\Im\lambda}\right|\right)=
 \frac 1{\lambda} O\left(\left|{e^{ia\lambda x}}\right|\right),
 \end{multline*}
and
 \begin{multline*}
 \int_0^xK_{12}(x,t)e^{ib\lambda t}\,dt=
 \frac 1{ib\lambda} \left(\left.K_{12}(x,t)e^{ib\lambda t}\right|_{t=0}^x-
 \int_0^x \frac{\partial }{\partial t}K_{12}(x,t)e^{ib\lambda t}\,dt\right) \\
 =\frac 1{ib\lambda} O(1)= \frac 1{\lambda}O\left(\left|{e^{ia\lambda x}}\right|\right),
 \end{multline*}
 since $\Im \lambda>0$ and $a<0<b$.

 Therefore from the formula \eqref{l1.1} one gets
 \begin{equation*}
 \f_{01}(x;\lambda)=\exp(a\lambda ix)+\tfrac1{\lambda} O(\exp(a\lambda ix)).
 \end{equation*}

 All the other estimates in Lemma \ref{lem.2} are proved in a  similar manner.
 \end{proof}

 %
 %

In the next lemma some estimates for the growth of the Wronski
determinant are presented.

 \begin{lemma} \label{lem.Wr}
Let $y_1(x;\lambda)$ and $y_2(x;\lambda)$ be two linearly
independent solutions of the system \eqref{3}. Then the Wronski
determinant
   \begin{equation} \label{Wr}
   W(x;\lambda)=\det
   \begin{pmatrix}
   y_{11}(x;\lambda) & y_{21}(x;\lambda) \\
   y_{12}(x;\lambda) & y_{22}(x;\lambda)
   \end{pmatrix}
   \end{equation}
admits the following  estimate
   \begin{equation} \label {lemWr.1}
   W(x;\lambda)=(1+o(1))\exp((a+b)\lambda i)W(0;\lambda).
   \end{equation}
 \end{lemma}
 \begin{proof}
It was proved in \cite{Mal+Or} that in $\mathbb{C}_{\pm}$ the system
\eqref{3} has two linearly independent solutions
$\varepsilon_1(x,\lambda)$ and $\varepsilon_2(x,\lambda)$ satisfying
the estimates
   \begin{equation} \label {lemWr.2}
   \varepsilon_1(x,\lambda)=\binom{(1+o(1))exp(a\lambda ix)}{o(1)exp(a\lambda ix)}
   \quad \text{ and } \quad
   \varepsilon_2(x,\lambda)=\binom{o(1)exp(b\lambda ix)}{(1+o(1))exp(b\lambda
   ix)}
   \end{equation}
for $\lambda \in \mathbb{C}_{\pm}$.

Since $y_j(x;\lambda)$ is a linear combination of
$\varepsilon_1(x,\lambda)$ and $\varepsilon_2(x,\lambda)$, one has
    \begin{multline}
    \label{lemWr.3}
   \begin{pmatrix}
   y_{11}(x;\lambda) & y_{21}(x;\lambda) \\
   y_{12}(x;\lambda) & y_{22}(x;\lambda)
   \end{pmatrix}
   = \\
   \begin{pmatrix}
   \varepsilon_{11}(x;\lambda) & \varepsilon_{21}(x;\lambda) \\
   \varepsilon_{12}(x;\lambda) & \varepsilon_{22}(x;\lambda)
   \end{pmatrix}
   \begin{pmatrix}
   \varepsilon_{11}(0;\lambda) & \varepsilon_{21}(0;\lambda) \\
   \varepsilon_{12}(0;\lambda) & \varepsilon_{22}(0;\lambda)
   \end{pmatrix}
     ^{-1}
   \begin{pmatrix}
   y_{11}(0;\lambda) & y_{21}(0;\lambda) \\
   y_{12}(0;\lambda) & y_{22}(0;\lambda)
   \end{pmatrix}.
   \end{multline} 
It follows that
   \begin{multline} \label{lemWr.4}
   W(x;\lambda)=\det
   \begin{pmatrix}
   \varepsilon_{11}(x;\lambda) & \varepsilon_{21}(x;\lambda) \\
   \varepsilon_{12}(x;\lambda) & \varepsilon_{22}(x;\lambda)
   \end{pmatrix}
     \det
   \begin{pmatrix}
   \varepsilon_{11}(0;\lambda) & \varepsilon_{21}(0;\lambda) \\
   \varepsilon_{12}(0;\lambda) & \varepsilon_{22}(0;\lambda)
   \end{pmatrix}
     ^{-1}
   W(0;\lambda) \\
   =(1+o(1))\exp((a+b)\lambda i)W(0;\lambda)
   \end{multline}
which gives the required estimate \eqref{lemWr.1}.
          \end{proof}

 %
 %

 The function $\chi(\lambda)$ defined by
 \begin{equation}
 \label{27}
 \chi(\lambda):=\det
 \begin{pmatrix}
 Q_{11}(\lambda) & Q_{12}(\lambda) \\
 Q_{21}(\lambda) & Q_{22}(\lambda)
 \end{pmatrix},
 \end{equation}
 where
 \begin{equation}
 \begin{split}
 Q_{11}(\lambda)=P_{11}(\lambda)+P_{13}(\lambda)\f_{01}(1;\lambda)+P_{14}(\lambda)\f_{02}(1;\lambda),     \\
 Q_{12}(\lambda)=P_{12}(\lambda)+P_{13}(\lambda)\psi_{01}(1;\lambda)+P_{14}(\lambda)\psi_{02}(1;\lambda), \\
 Q_{21}(\lambda)=P_{21}(\lambda)+P_{23}(\lambda)\f_{01}(1;\lambda)+P_{24}(\lambda)\f_{02}(1;\lambda),     \\
 Q_{22}(\lambda)=P_{22}(\lambda)+P_{23}(\lambda)\psi_{01}(1;\lambda)+P_{24}(\lambda)\psi_{02}(1;\lambda),
 \end{split}
 \end{equation}
 is said to be \textit{the characteristic function}
 of the problem \eqref{3}, \eqref{5}.
 This definition is motivated by the next result.

 \begin{proposition}
 \label{SEAF}
 The number $\lambda_0\in\dC$ is an eigenvalue
 of the operator associated to the problem \eqref{3}, \eqref{5}
 if and only if $\chi(\lambda_0)=0$.
 Moreover, the functions $\vek{\omega_1}(x;\lambda)$ and $\vek{\omega_2}(x;\lambda)$
 given by
 \begin{equation}
 \label{28a}
 \vek{\omega_1}(x;\lambda)=Q_{12}(\lambda)\f_0(x;\lambda)
      -Q_{11}(\lambda)\psi_0(x;\lambda)
 \end{equation}
 and
 \begin{equation}
 \label{29a}
 \vek{\omega_2}(x;\lambda)=Q_{22}(\lambda)\f_0(x;\lambda)
  -Q_{21}(\lambda)\psi_0(x;\lambda)
 \end{equation}
 are the eigenfunctions corresponding to the eigenvalue $\lambda_0,$ or,
 one has $\vek{\omega_j}(x;\lambda)\equiv 0$.
 Moreover, all the eigenfunctions and associate functions corresponding to
 the eigenvalue $\lambda_0$ are the nonzero functions of the form
 \begin{equation}
 \label{32a}
 \left.\frac 1{k!}\frac{\partial^k}{\partial\lambda^k}\omega_j(x;\lambda)
 \right|_{\lambda=\lambda_0},\quad\text{ where } 0\le k<p_j,\quad j=1,2.
 \end{equation}
 \end{proposition}

 \begin{proof}

 It follows from \eqref{8} with $\alpha=0$ that for all $\lambda\in\mathbb{C}$ the function
 \begin{equation}
 \label{28}
 \begin{split}
 \vek{\omega_1}(x;\lambda)=\,&
 (P_{12}(\lambda)+P_{13}(\lambda)\psi_{01}(1;\lambda)+
               P_{14}(\lambda)\psi_{02}(1;\lambda))\f_0(x;\lambda) \\
  &-(P_{11}(\lambda)+P_{13}(\lambda)\f_{01}(1;\lambda)+
               P_{14}(\lambda)\f_{02}(1;\lambda))\psi_0(x;\lambda)
 \end{split}
 \end{equation}
 is a solution of the first equation in \eqref{5}.
 Moreover, since
 \begin{equation}
 \label{31}
 P_{21}(\lambda)\omega_{11}(0,\lambda)+P_{22}(\lambda)\omega_{12}(0,\lambda)
  +P_{23}(\lambda)\omega_{11}(1,\lambda)+P_{24}(\lambda)\omega_{12}(1,\lambda)=-\chi(\lambda),
 \end{equation}
 $\vek{\omega_1}(x;\lambda_0)$ is a solution of
 the second equation in \eqref{5},
 if $\lambda_0$ is a root of $\chi(\lambda)$.
 Similarly, for all $\lambda\in\mathbb{C}$ the function
 \begin{equation}
 \label{29}
 \begin{split}
 \vek{\omega_2}(x;\lambda)=\,&
 (P_{22}(\lambda)+P_{23}(\lambda)\psi_{01}(1;\lambda)+
               P_{24}(\lambda)\psi_{02}(1;\lambda))\f_0(x;\lambda) \\
  &-(P_{21}(\lambda)+P_{23}(\lambda)\f_{01}(1;\lambda)+
                P_{24}(\lambda)\f_{02}(1;\lambda))\psi_0(x;\lambda).
 \end{split}
 \end{equation}
 is a solution of the second equation in \eqref{5} and since
 \begin{equation}
 \label{31b}
 P_{11}(\lambda)\omega_{21}(0,\lambda)+P_{12}(\lambda)\omega_{22}(0,\lambda)
  +P_{13}(\lambda)\omega_{21}(1,\lambda)+P_{14}(\lambda)\omega_{22}(1,\lambda)=\chi(\lambda),
 \end{equation}
 $\vek{\omega_j}(x;\lambda_0)$ is a solution of
 the first equation in \eqref{5} too,
 if $\lambda_0$ is a root of $\chi(\lambda)$.

 If, in addition, $\lambda_0$ is a root of $\chi(\lambda)$ of the order $p_0=p_0^{(1)}+p_0^{(2)}$,
 then the operator determined by \eqref{3}, \eqref{5} has
 precisely $p$ eigenfunctions and associate functions
 corresponding to $\lambda_0$.
 In fact, it follows from \eqref{31}, \eqref{31b}
 that all nonzero functions
 \begin{equation}
 \label{32}
 \left.\frac 1{k!}\frac{\partial^k}{\partial\lambda^k}\omega_j(x;\lambda)
 \right|_{\lambda=\lambda_0},\quad\text{ where } \min(p_0^{(1)},p_0^{(2)}) \le k<p_0,\quad j=1,2,
 \end{equation}
 are eigenfunctions and associate functions corresponding to
 the eigenvalue $\lambda_0$.
 \end{proof}

 The completeness result can now be stated as follows.

 \begin{theorem}
 \label{theo.1}
 Let $P_{ij}$ (i=1,2 ; j=1,2,3,4) be polynomials,
 let the rank of the polynomial matrix
 \begin{equation}
 P(\lambda)=
 \begin{pmatrix}
 P_{11}(\lambda) & P_{12}(\lambda) & P_{13}(\lambda) & P_{14}(\lambda) \\
 P_{21}(\lambda) & P_{22}(\lambda) & P_{23}(\lambda) & P_{24}(\lambda)
 \end{pmatrix}
 \end{equation}
 be equal to 2 for all $\lambda\in\dC$, and let
 \begin{equation}
 \label{23}
 \deg J_{14}=\deg J_{32}\ge \max\{\deg J_{13},\deg J_{42},M\},
 \end{equation}
 where $M=\max\{\,\deg P_{ij}:\ i\in\{1,2\};\ j\in\{1,2,3,4\}\, \}$ and
 \begin{equation} \label{t3.2}
 J_{ij}=\det\begin{pmatrix}
 P_{1i} & P_{1j} \\
 P_{2i} & P_{2j}
 \end{pmatrix} ,\quad i,j=\{1,2,3,4\}.
 \end{equation}
 Then the SEAF of the problem \eqref{3}, \eqref{5} is complete
 in $L^2[0,1]\oplus L^2[0,1]$.

 Moreover, let the set $\Phi$, which consists of $N:=\deg J_{14}-M$
 eigenfunctions and associate functions, satisfy the following condition:

 If $\Phi$ contains either an eigenfunction or an associate function
 corresponding to an eigenvalue $\lam_k$, then it also contains all
 the associate functions of higher order corresponding to
 the same eigenvalue.

 Then the SEAF of the problem \eqref{3}, \eqref{5} without the set $\Phi$
 is also complete in the space $L^2[0,1]\oplus L^2[0,1].$
 \end{theorem}

 \begin{proof}
 Suppose that the SEAF of the problem \eqref{3}, \eqref{5} without
 the set $\Phi$ is not complete in the space $L^2[0,1]\oplus L^2[0,1]$.
 Then there exists a nonzero vector function
 $\vec f(x)=(f_1(x),f_2(x))^\top$,
 which is orthogonal to the SEAF of the problem \eqref{3}, \eqref{5}
 (possibly, excluding functions from the set $\Phi$).
 Define
 \begin{equation} \label {35}
 \widetilde{F_j}(\lam):=\left\langle\vek \omega_j(x;\lam),\vek f(x)\right\rangle
 =\int_0^1(\omega_{j1}(x;\lam)\overline{f_1(x)}
  +\omega_{j2}(x;\lam)\overline{f_2(x)})\,dx.
 \end{equation}
 Clearly, $\widetilde{F_j}(\lam)$ is an entire function.
 If $\lambda_s$ is an eigenvalue of multiplicity $p_s=p_s^{(1)}+p_s^{(2)}$ and the set $\Phi$
 contains neither an eigenfunction nor an associate function corresponding
 to $\lambda_s$, then it follows from Proposition \ref{SEAF} that $\lambda_s$ is
 a root of $\widetilde{F_j}(\lam)$ of order $p_s^{(j)}$, $j=1,2$.

 If $\Phi$ contains $k$ eigenfunctions or associate functions
 corresponding to the eigenvalue $\lambda_s$, then $\lambda_s$ is a root
 of $\widetilde{F_j}(\lam)$ of order greater
 than or equal to $p_s^{(j)}-k$, $j=1,2$.

 Let $\Phi\ne\emptyset$ and denote by $\Lambda$ the set of all eigenvalues of the problem
 \eqref{3}, \eqref{5}, such that the corresponding eigenfunctions
 (or associate functions) belong to the set $\Phi$. For each
 $\lambda_s\in\Lambda$ denote by $q_s$ the number of eigenfunctions
 and associate functions in $\Phi$ corresponding to $\lambda_s$.
 Define
 \begin{equation}
 \label{36}
 \Pi(\lam)=\prod_{\lam_s\in\Lambda}(\lam-\lam_s)^{q_s}.
 \end{equation}

 Let $\lambda_k$ be an eigenvalue of the problem \eqref{3}, \eqref{5} of
 multiplicity $p_k$. Then $\lambda_k$ is a zero of the product
 $\Pi(\lam) \widetilde F_j(\lam)$ at least of order $p_k$.
 Consequently, the functions
 \begin{equation}
 \label{37}
 F_j(\lam)=\frac{\Pi(\lam)\widetilde{F_j}(\lam)}{\chi(\lam)}
 \end{equation}
 are entire. Next an estimate for these functions will be derived.

 One can rewrite $\chi(\lambda)$ as follows
 \begin{equation} \label {38}
 \begin{split}
 \chi(\lambda)
  &=J_{12}+J_{13}\psi_{01}(1;\lambda)+J_{14}\psi_{02}(1;\lambda)
    +J_{32}\f_{01}(1;\lambda)\\
 & \qquad  +J_{42}\f_{02}(1;\lambda)+J_{34} \det
 \begin{pmatrix}
 \f_{01}(1;\lambda) & \psi_{01}(1;\lambda) \\
 \f_{02}(1;\lambda) & \psi_{02}(1;\lambda)
 \end{pmatrix} \\
 &=J_{12}+J_{13}\psi_{01}(1;\lambda)+J_{14}\psi_{02}(1;\lambda)+
   J_{32}\f_{01}(1;\lambda) \\
 &\qquad +J_{42}\f_{02}(1;\lambda)+J_{34}
  \left(1+O\left(\tfrac 1{\Im\lambda}\right)\right)\exp((a+b)\lambda i).
 \end{split}
 \end{equation}
 Then one obtains from \eqref{11}, \eqref{12}, and the assumption
 \eqref{23} the following estimates for $\chi(\lambda)$:
 \begin{equation} \label {42}
 \chi(\lambda)=(1+O(\tfrac 1{\Im\lambda}))J_{32}\exp(a\lambda i),
 \quad \lambda\in\dC^+;
 \end{equation}
 \begin{equation} \label {43}
 \chi(\lambda)=(1+O(\tfrac 1{\Im\lambda}))J_{14}\exp(b\lambda i),
 \quad \lambda\in\dC^-.
 \end{equation}
 Moreover, the definition of $\vek \omega_j(x;\lambda)$
 (cf. \eqref{28} and \eqref{29}) implies that
 \begin{equation}
 \label{47}
 \begin{split}
 \vek \omega_j(x;\lambda)
 &=-P_{j1}(\lambda)\psi_0(x;\lambda)+P_{j2}(\lambda)\f_0(x;\lambda) \\
 &\qquad +(1+O\left(\tfrac 1{\Im\lambda}\right))\exp((a+b)\lambda i)
          (-P_{j3}(\lambda)\psi_{1}(x;\lambda)+P_{j4}(\lambda)\f_{1}(x;\lambda)).
 \end{split}
 \end{equation}
If $\lambda\in\dC^+$, then \eqref{47} and the estimate \eqref{11}
imply
 \begin{equation} \label {48}
 \begin{split}
 \omega_{jk}(x;\lambda)
 &=(O(P_{j1}(\lambda))+O(P_{j2}(\lambda)))\exp(a\lambda ix) \\
 &\qquad +(O(P_{j3}(\lambda))+O(P_{j4}(\lambda)))\exp(a\lambda i)\exp(b\lambda ix).
 \end{split}
 \end{equation}
 By using the Cauchy--Schwartz inequality one gets
 \[
 \int_0^1|f_k(x)\exp(a\lambda ix)|\,dx
   =O\left(\frac{\exp(a\lambda i)}{\sqrt{|\Im\lambda|}}\right), \quad
 \int_0^1|f_k(x)\exp(b\lambda ix)|\,dx=O\left(\frac{1}{\sqrt{|\Im\lambda|}}\right),
 \]
 and consequently there exists a constant $c_1>0$, such that, for $\Im\lambda>c_1$ and $a<0<b$,
 \begin{equation} \label{50}
 \widetilde{F}_j(\lambda)
  =O\left(\max|P_{jk}(\lambda)|\frac{\exp(a\lambda i)}{\sqrt{|\Im\lambda|}}\right)
  =O\left(\frac{\lambda^{M}}{\sqrt{|\Im\lambda|}}\right)e^{a\lambda i}.
 \end{equation}
 Similarly, there exists a constant $c_2<0$, such that, for $\Im\lambda<c_2$,
 \begin{equation} \label{51}
 \widetilde{F}_j(\lambda)=O\left(\frac{\lambda^{M}}{\sqrt{|\Im\lambda|}}\right)e^{b\lambda i}.
 \end{equation}
 From \eqref{42}, \eqref{43}, \eqref{50}, \eqref{51}, and the assumption \eqref{23}
 one obtains finally the estimate
 \begin{equation} \label{52}
 F_j(\lambda)=O\left(\frac1{\sqrt{|\Im\lambda|}}\right),
 \quad |\Im\lambda|>c.
 \end{equation}
 By applying Phragmen--Lindel\"{o}f theorem for a strip one concludes that $F_j(\lambda)\equiv 0$.
 Consequently, $\widetilde{F}_j(\lambda)\equiv 0$,
 i.e. $\vec f(x)$ is orthogonal to
 $\vek\omega_1(x;\lambda)$ and $\vek\omega_2(x;\lambda)$ for all $\lambda$.
 However, the functions $\vek\omega_1(x;\lambda)$ and $\vek\omega_2(x;\lambda)$
 for all $\lambda$ form a fundamental system of solutions
 of the equation \eqref{3} if $\lambda$ is not an eigenvalue.
 Since the set of eigenvalues coincides with the set of
 all roots of $\chi(\lambda)$, this set is discrete.
 This implies that $\vek{f}(x)$ is orthogonal to all solutions of
 the equation \eqref{3}, so that $\vek{f}(x)\equiv 0$.

 Therefore, there is no nontrivial function $\vek{f}(x)$
 orthogonal to the SEAF of the problem \eqref{3}--\eqref{5}
 (maybe without the set $\Phi$).
 \end{proof}



 \begin{theorem} \label {theo.2}
 Let $P_{ij}$ (i=1,2 ; j=0,1,\dots,9) be polynomials,
 let the rank of the matrix
 \begin{equation}
 \begin{pmatrix}
 P_{10}(\lambda) & P_{11}(\lambda) & \dots & P_{19}(\lambda) \\
 P_{20}(\lambda) & P_{21}(\lambda) & \dots & P_{29}(\lambda)
 \end{pmatrix}
 \end{equation}
 be equal to 2 for all $\lambda\in\dC$, and let
 \begin{equation} \label {54}
 \deg J_{03}=\deg J_{12}=M,
 \end{equation}
 where
 \begin{equation}
 J_{ij}=\det\begin{pmatrix}
 P_{1i} & P_{1j} \\
 P_{2i} & P_{2j}
 \end{pmatrix} ,\quad i,j=0,1,\dots,9,
 \end{equation}
 and $M=\max\{\,\deg P_{ij}:\, i\in\{1,2\};\, j\in\{0,1,\dots,9\}\,\}$.
 Then the SEAF of problem \eqref{3}, \eqref{6} is complete in
 $L^2[0,1]\oplus L^2[0,1]$.

 Moreover, let the set $\Phi$, which consists of $M$
 eigenfunctions and associate functions, satisfy
 the following condition:

 If $\Phi$ contains either an eigenfunction or an associate function
 corresponding to an eigenvalue $\lam_k$, then it also contains
 all the associate functions of higher order corresponding to
 the same eigenvalue.

 Then the SEAF of the problem \eqref{3}, \eqref{6} without the set $\Phi$
 is also complete in the space $L^2[0,1]\oplus L^2[0,1]$.
 \end{theorem}

 \begin{proof}

 The proof of this theorem is similar to the proof of Theorem \ref{theo.1}.

 As in Theorem \ref{theo.1} one considers the characteristic function
 and functions $\vek\omega_1(x;\lambda)$ and $\vek\omega_2(x;\lambda)$,
 but in this case these functions will be defined by other formulas.

 Let $\vek\omega(x;\lambda)$ be an arbitrary solution of the system
 \eqref{3}. Then $\vek\omega(x;\lambda)$ may be written in the form
 \begin{equation*}
 \vek{\omega}(x;\lambda)=A\vek{\f_0}(x;\lambda)+B\vek{\psi_0}(x;\lambda).
 \end{equation*}
 It follows that $\vek\omega(x;\lambda)$ is a solution of the problem
 \eqref{3}, \eqref{6} if and only if
 \begin{equation} \label{55}
 \begin{cases}
 \begin{split}
 P_{10}(\lam)A^2+P_{11}(\lam)B^2+
      P_{12}(\lam)(A\f_{01}(\tfrac 12;\lambda)+B\psi_{01}(\tfrac 12;\lambda))^2+ \\
 P_{13}(\lam)(A\f_{02}(\tfrac 12;\lambda)+B\psi_{02}(\tfrac 12;\lambda))^2+
      P_{14}(\lam)AB+                                                  \\
 P_{15}(\lam)A(A\f_{01}(\tfrac 12;\lambda)+B\psi_{01}(\tfrac 12;\lambda))+
      P_{16}(\lam)A(A\f_{02}(\tfrac 12;\lambda)+B\psi_{02}(\tfrac 12;\lambda))+  \\
 P_{17}(\lam)B(A\f_{01}(\tfrac 12;\lambda)+B\psi_{01}(\tfrac 12;\lambda))+
      P_{18}(\lam)B(A\f_{02}(\tfrac 12;\lambda)+B\psi_{02}(\tfrac 12;\lambda))+  \\
 P_{19}(\lam)(A\f_{01}(\tfrac 12;\lambda)+B\psi_{01}(\tfrac 12;\lambda))
      (A\f_{02}(\tfrac 12;\lambda)+B\psi_{02}(\tfrac 12;\lambda))=0,             \\
 \end{split}                                                      \\   \\
 \begin{split}
 P_{20}(\lam)A^2+P_{21}(\lam)B^2+
      P_{22}(\lam)(A\f_{01}(\tfrac 12;\lambda)+B\psi_{01}(\tfrac 12;\lambda))^2+ \\
 P_{23}(\lam)(A\f_{02}(\tfrac 12;\lambda)+B\psi_{02}(\tfrac 12;\lambda))^2+
      P_{24}(\lam)AB+                                                  \\
 P_{25}(\lam)A(A\f_{01}(\tfrac 12;\lambda)+B\psi_{01}(\tfrac 12;\lambda))+
      P_{26}(\lam)A(A\f_{02}(\tfrac 12;\lambda)+B\psi_{02}(\tfrac 12;\lambda))+  \\
 P_{27}(\lam)B(A\f_{01}(\tfrac 12;\lambda)+B\psi_{01}(\tfrac 12;\lambda))+
      P_{28}(\lam)B(A\f_{02}(\tfrac 12;\lambda)+B\psi_{02}(\tfrac 12;\lambda))+  \\
 P_{29}(\lam)(A\f_{01}(\tfrac 12;\lambda)+B\psi_{01}(\tfrac 12;\lambda))
      (A\f_{02}(\tfrac 12;\lambda)+B\psi_{02}(\tfrac 12;\lambda))=0.
 \end{split}
 \end{cases}
 \end{equation}

 The system \eqref{55} may be rewritten in the form
 \begin{equation}
 \begin{cases}
 Q_{11}(\lambda)A^2+Q_{12}(\lambda)AB+Q_{13}(\lambda)B^2=0 \\
 Q_{21}(\lambda)A^2+Q_{22}(\lambda)AB+Q_{23}(\lambda)B^2=0,
 \end{cases}
 \end{equation}
 where
 \begin{equation}
 \begin{split}
 Q_{11}=
     & P_{10}(\lam)+
          P_{12}(\lam)\f_{01}^2(\tfrac 12;\lambda)+
          P_{13}(\lam)\f_{02}^2(\tfrac 12;\lambda)+ \\
     & +P_{15}(\lam)\f_{01}(\tfrac 12;\lambda)+
          P_{16}(\lam)\f_{02}(\tfrac 12;\lambda)+
          P_{19}(\lam)\f_{01}(\tfrac 12;\lambda)\f_{02}(\tfrac 12;\lambda),
  \\ \\
 Q_{12}=
     & 2P_{12}(\lam)\f_{01}(\tfrac 12;\lambda)\psi_{01}(\tfrac 12;\lambda)+
           2P_{13}(\lam)\f_{02}(\tfrac 12;\lambda)\psi_{02}(\tfrac 12;\lambda)+
           P_{14}(\lam)+ \\
     & +P_{15}(\lam)\psi_{01}(\tfrac 12;\lambda)+
           P_{16}(\lam)\psi_{02}(\tfrac 12;\lambda)+
           P_{17}(\lam)\f_{01}(\tfrac 12;\lambda)+ \\
     & +P_{18}(\lam)\f_{02}(\tfrac 12;\lambda)+
           P_{19}(\lam)(\f_{01}(\tfrac 12;\lambda)\psi_{02}(\tfrac 12;\lambda)+
           \psi_{01}(\tfrac 12;\lambda)\f_{02}(\tfrac 12;\lambda)),
  \\ \\
 Q_{13}=
     & P_{11}(\lam)+
           P_{12}(\lam)\psi_{01}^2(\tfrac 12;\lambda))+
           P_{13}(\lam)\psi_{02}^2(\tfrac 12;\lambda)+ \\
     & +P_{17}(\lam)\psi_{01}(\tfrac 12;\lambda))+
           P_{18}(\lam)\psi_{02}(\tfrac 12;\lambda)+
           P_{19}(\lam)\psi_{01}(\tfrac 12;\lambda)\psi_{02}(\tfrac 12;\lambda) ;
  \\ \\ \\
 Q_{21}=
     & P_{20}(\lam)+
          P_{22}(\lam)\f_{01}^2(\tfrac 12;\lambda)+
          P_{23}(\lam)\f_{02}^2(\tfrac 12;\lambda)+ \\
     & +P_{25}(\lam)\f_{01}(\tfrac 12;\lambda)+
          P_{26}(\lam)\f_{02}(\tfrac 12;\lambda)+
          P_{29}(\lam)\f_{01}(\tfrac 12;\lambda)\f_{02}(\tfrac 12;\lambda),
  \\ \\
 Q_{22}=
     & 2P_{22}(\lam)\f_{01}(\tfrac 12;\lambda)\psi_{01}(\tfrac 12;\lambda)+
           2P_{23}(\lam)\f_{02}(\tfrac 12;\lambda)\psi_{02}(\tfrac 12;\lambda)+
           P_{24}(\lam)+ \\
     & +P_{25}(\lam)\psi_{01}(\tfrac 12;\lambda)+
           P_{26}(\lam)\psi_{02}(\tfrac 12;\lambda)+
           P_{27}(\lam)\f_{01}(\tfrac 12;\lambda)+ \\
     & +P_{28}(\lam)\f_{02}(\tfrac 12;\lambda)+
           P_{29}(\lam)(\f_{01}(\tfrac 12;\lambda)\psi_{02}(\tfrac 12;\lambda)+
           \psi_{01}(\tfrac 12;\lambda)\f_{02}(\tfrac 12;\lambda)),
  \\ \\
 Q_{23}=
     & P_{21}(\lam)+
           P_{22}(\lam)\psi_{01}^2(\tfrac 12;\lambda))+
           P_{23}(\lam)\psi_{02}^2(\tfrac 12;\lambda)+ \\
     & +P_{27}(\lam)\psi_{01}(\tfrac 12;\lambda))+
           P_{28}(\lam)\psi_{02}(\tfrac 12;\lambda)+
           P_{29}(\lam)\psi_{01}(\tfrac 12;\lambda)\psi_{02}(\tfrac 12;\lambda)
 \end{split}
 \end{equation}

 It is well known (see, for example, \cite{VdW}), that a system of
 two quadratic equations has a nonzero solution if and only if, the
 resultant is equal to 0. Therefore, $\lambda_0$ is an eigenvalue of
 the problem \eqref{3}, \eqref{6} if and only if $\chi(\lambda)=0$, where
 \begin{equation} \label {57}
 \chi(\lambda)=\det
 \begin{pmatrix}
 Q_{11} & Q_{12} & Q_{13} & 0 \\
 0 & Q_{11} & Q_{12} & Q_{13} \\
 Q_{21} & Q_{22} & Q_{23} & 0 \\
 0 & Q_{21} & Q_{22} & Q_{23}
 \end{pmatrix}=D_{13}^2-D_{12}D_{23}
 \end{equation}
 with
 \begin{equation} \label{58}
 D_{ij}=\det
 \begin{pmatrix}
 Q_{1i} & Q_{1j} \\
 Q_{2i} & Q_{2j}
 \end{pmatrix}.
 \end{equation}

 Moreover, the multiplicity of $\lambda_0$ as a zero of the function
 $\chi(\lambda)$ is equal to the number of eigenfunctions and associate
 functions corresponding to the eigenvalue $\lambda_0$.

 Introduce the Wronski determinant
 \begin{equation*}
 W(x;\lambda)=\det
 \begin{pmatrix}
 \f_{01}(x;\lambda) & \psi_{01}(x;\lambda) \\
 \f_{02}(x;\lambda) & \psi_{02}(x;\lambda)
 \end{pmatrix}.
 \end{equation*}
 Then $\chi(\lambda)$ may be transformed to a polynomial of degree 4
 with the arguments $\f_{01}(\tfrac 12;\lambda),$ $\psi_{01}(\tfrac
 12;\lambda),\ \f_{02}(\tfrac 12;\lambda),\ \psi_{02}(\tfrac 12;\lambda)$, and
$W(\tfrac 12,\lambda)$, having all the coefficients of form
$J_{ij}J_{kl}$.
 In particular, the coefficient of $\f_{0,1}^4(\tfrac 12;\lambda)$ is
 equal to $J_{12}^2$, and the coefficient of $\psi_{0,2}^4(\tfrac
 12;\lambda)$ is equal to $J_{03}^2$.
Therefore, from the condition \eqref{54}, the estimates \eqref{11},
 \eqref{12}, and the following estimate (cf. Lemma~\ref{lem.Wr})
 \begin{equation} \label {59}
 W(x;\lambda)=\left(1+O\left(\tfrac 1\lambda\right)\right)\exp((a+b)\lambda i),
 \end{equation}
 one can derive the following estimates for the characteristic function $\chi(\lambda)$:
 \begin{equation} \label {60}
 \chi(\lambda)=(1+O(\tfrac 1{\Im\lambda}))J_{12}^2\exp(2a\lambda i)
 \quad \text{ for }\lambda\in\dC^+;
 \end{equation}
 and
 \begin{equation} \label {61}
 \chi(\lambda)=(1+O(\tfrac 1{\Im\lambda}))J_{01}^2\exp(2b\lambda i)
 \quad \text{ for }\lambda\in\dC^-.
 \end{equation}

 Now, introduce the functions
 \begin{equation} \label {62}
 \vek\omega_1(x;\lambda):=D_{13}\f_0(x;\lambda)-D_{12}\psi_0(x;\lambda) \quad \text{ and }\
 \vek\omega_2(x;\lambda):=D_{23}\f_0(x;\lambda)-D_{13}\psi_0(x;\lambda).
 \end{equation}
 The function $\vek\omega_1(x;\lambda)$ satisfies
 the boundary conditions \eqref{6} if and only if
 \begin{equation} \label{63}
 \begin{cases}
 \Gamma_1(\lambda):=Q_{11}(\lambda)D_{13}^2(\lambda)-Q_{12}(\lambda)D_{13}(\lambda)D_{12}(\lambda)+Q_{13}(\lambda)D_{12}^2(\lambda)=0 \\
 \Gamma_2(\lambda):=Q_{21}(\lambda)D_{13}^2(\lambda)-Q_{22}(\lambda)D_{13}(\lambda)D_{12}(\lambda)+Q_{23}(\lambda)D_{12}^2(\lambda)=0 \\
 \end{cases}.
 \end{equation}
 Observe that
 \begin{equation}
 \Gamma_1(\lambda)=Q_{11}(\lambda)\chi(\lambda) \quad\text{and }\
 \Gamma_2(\lambda)=Q_{21}(\lambda)\chi(\lambda).
 \end{equation}
 Hence, if $\lambda_0$ is an eigenvalue of multiplicity $p$, then
 \begin{equation} \label {64}
 \left.\frac{\partial^k}{\partial\lambda^k}\Gamma_1(x;\lambda)\right|_{\lambda=\lambda_0}=0
 \quad\text{ and }
 \left.\frac{\partial^k}{\partial\lambda^k}\Gamma_2(x;\lambda)\right|_{\lambda=\lambda_0}=0
 \quad\text{ for all } k<p.
 \end{equation}
 Therefore, in this case, all nonzero functions
 \begin{equation} \label {65}
 \left.\frac{\partial^k}{\partial\lambda^k}\vec\omega_1(x;\lambda)\right|_{\lambda=\lambda_0}=0,
 \quad \text{ with } k<p,
 \end{equation}
 are eigenfunctions and associate functions, corresponding to the eigenvalue $\lambda_0$.
 Similarly all nonzero functions given by
 \begin{equation} \label {66}
 \left.\frac{\partial^k}{\partial\lambda^k}\vec\omega_2(x;\lambda)\right|_{\lambda=\lambda_0}=0,
 \quad \text{ with } k<p,
 \end{equation}
 are eigenfunctions and associate functions, corresponding to the eigenvalue $\lambda_0,$ too.

 Suppose that the SEAF of the problem \eqref{3}, \eqref{6} without the set $\Phi$ is
 not complete in the space $L^2[0,1]\oplus L^2[0,1]$. Then there exists
 a nonzero vector-function
 \begin{equation}
 \vec f(x)=\binom {f_1(x)}{f_2(x)}
 \end{equation}
 which is orthogonal to the SEAF of problem \eqref{3}--\eqref{6}
 (possibly excluding the functions from the set $\Phi$).

 Just as in the proof of Theorem \ref{theo.1} introduce the
 functions $\tilde{F}_j(\lambda)$ and $\Pi(\lambda)$ by the formulae \eqref{35}
 and \eqref{36}. Then, as before, the functions
 \begin{equation} \label {67}
 F_j(\lam)=\frac{\Pi(\lam)\tilde{F_j}(\lam)}{\chi(\lam)}
 \end{equation}
 are entire. Let
 \begin{equation*}
 \vek{g}(x)=(c_1\psi_{01}(\tfrac 12;\lambda)+c_2\psi_{02}(\tfrac 12;\lambda))\vek{\f_{0}}(x;\lambda)-
 (c_1\f_{01}(\tfrac 12;\lambda)+c_2\f_{02}(\tfrac 12;\lambda))\vek{\psi_{0}}(x;\lambda),
 \end{equation*}
 where $c_1$ and $c_2$ are arbitrary complex coefficients.
 From \eqref{59} one gets the estimate
 \begin{equation*}
 \vek{g}(\tfrac 12)=\binom
 {(-c_1+O\left(\tfrac 1\lambda\right))\exp((a+b)\lambda i)}
 {(c_2+O\left(\tfrac 1\lambda\right))\exp((a+b)\lambda i)}
 \end{equation*}
 and this implies that the function $\vek{g}(x)$ satisfies the estimate
 \begin{equation} \label {68}
 \vek{g}(x)=
 \left(1+O\left(\tfrac 1{\Im\lambda}\right)\right)\exp(\tfrac 12(a+b)\lambda i)
 (-c_1\vek{\psi_{\tfrac 12}}(x;\lambda)+c_2\vek{\f_{\tfrac 12}}(x;\lambda)),
 \end{equation}
 where the functions $\vek{\f_{\tfrac 12}}(x;\lambda)$ and $\vek{\psi_{\tfrac 12}}(x;\lambda)$
 are solutions of the Cauchy problem for the system \eqref{3} with the initial
 conditions
 \begin{equation} \label{69}
 {\f_{\tfrac 12 1}(\tfrac 12;\lambda)}={\psi_{\tfrac 12 2}(\tfrac 12;\lambda)}=1
 \quad \text{ and } \quad
 {\f_{\tfrac 12 2}(\tfrac 12;\lambda)}={\psi_{\tfrac 12 1}(\tfrac 12;\lambda)}=0.
 \end{equation}
As in Lemma \ref{lem.1} one can derive for these functions the
following estimates:

 if $\lambda\in\dC^+$ and $x>\tfrac 12$ then
 \begin{equation} \label{70}
 \begin{split}
 \f_{\tfrac 12 1}(x;\lambda)&=(1+O(\tfrac1{\Im\lambda}))\exp(a\lambda i(x-\tfrac 12 )), \\
 \f_{\tfrac 12 2}(x;\lambda)&=O(\tfrac1{\Im\lambda})\exp(a\lambda i(x-\tfrac 12 ));
  \\ \\
 \psi_{\tfrac 12 1}(x;\lambda)&=O(\tfrac1{\Im\lambda})\exp(a\lambda i(x-\tfrac 12 )), \\
 \psi_{\tfrac 12 2}(x;\lambda)&=O(\tfrac1{\Im\lambda})\exp(a\lambda i(x-\tfrac 12 ));
 \end{split}
 \end{equation}

 if $\lambda\in\dC^+$ and $x<\tfrac 12$ then
 \begin{equation} \label{71}
 \begin{split}
 \f_{\tfrac 12 1}(x;\lambda)&=O(\tfrac1{\Im\lambda})\exp(b\lambda i(x-\tfrac 12 )), \\
 \f_{\tfrac 12 2}(x;\lambda)&=O(\tfrac1{\Im\lambda})\exp(b\lambda i(x-\tfrac 12 ));
  \\ \\
 \psi_{\tfrac 12 1}(x;\lambda)&=O(\tfrac1{\Im\lambda})\exp(b\lambda i(x-\tfrac 12 )), \\
 \psi_{\tfrac 12 2}(x;\lambda)&=(1+O(\tfrac1{\Im\lambda}))\exp(b\lambda i(x-\tfrac 12 ));
 \end{split}
 \end{equation}

 if $\lambda\in\dC^-$ and $x>\tfrac 12$ then
 \begin{equation} \label{72}
 \begin{split}
 \f_{\tfrac 12 1}(x;\lambda)&=O(\tfrac1{\Im\lambda})\exp(b\lambda i(x-\tfrac 12 )), \\
 \f_{\tfrac 12 2}(x;\lambda)&=O(\tfrac1{\Im\lambda})\exp(b\lambda i(x-\tfrac 12 ));
  \\ \\
 \psi_{\tfrac 12 1}(x;\lambda)&=O(\tfrac1{\Im\lambda})\exp(b\lambda i(x-\tfrac 12 )), \\
 \psi_{\tfrac 12 2}(x;\lambda)&=(1+O(\tfrac1{\Im\lambda}))\exp(b\lambda i(x-\tfrac 12 ));
 \end{split}
 \end{equation}

 if $\lambda\in\dC^-$ and $x<\tfrac 12$ then
 \begin{equation} \label{73}
 \begin{split}
 \f_{\tfrac 12 1}(x;\lambda)&=(1+O(\tfrac1{\Im\lambda}))\exp(b\lambda i(x-\tfrac 12 )), \\
 \f_{\tfrac 12 2}(x;\lambda)&=O(\tfrac1{\Im\lambda})\exp(b\lambda i(x-\tfrac 12 ));
  \\ \\
 \psi_{\tfrac 12 1}(x;\lambda)&=O(\tfrac1{\Im\lambda})\exp(b\lambda i(x-\tfrac 12 )), \\
 \psi_{\tfrac 12 2}(x;\lambda)&=O(\tfrac1{\Im\lambda})\exp(a\lambda i(x-\tfrac 12
 )).
 \end{split}
 \end{equation}

 Using the formula \eqref{68} and the estimates
 \eqref{11}, \eqref{12}, \eqref{70} -- \eqref{73} one gets the following estimates
 for the functions in \eqref{62}
 \begin{equation} \label {74}
 \vek{\omega_i}(x;\lam)=O(\lambda^M\exp(ia\lam (x+1))),\qquad \text{ if } \Im\lam>0
 \end{equation}
 and
 \begin{equation} \label {75}
 \vek{\omega_i}(x;\lam)=O(\lambda^M\exp(ib\lam (x+1))),\qquad \text{ if } \Im\lam<0.
 \end{equation}

 From the estimates  \eqref{60}, \eqref{61}, \eqref{74}, \eqref{75}
 one gets finally the estimates
 \begin{equation}
 F_i(\lambda)=O(\frac 1{\sqrt{|\Im\lambda|}}),\qquad\ |\Im\lambda|>C,
 \end{equation}
 where $C>0$ is a constant.
 Then, by Phragmen--Lindel\"{o}f theorem for a strip, one again concludes that $F_j(\lambda)\equiv 0$,
 and therefore $\tilde{F}_j(\lambda)\equiv 0$, i.e. $\vec f(x)$ is orthogonal to
 $\vek\omega_1(x;\lambda)$ and $\vek\omega_2(x;\lambda)$ for all $\lambda$.

 Observe, that if $\chi(\lambda)\ne 0$ then the functions $\vek\omega_1(x;\lambda)$ and $\vek\omega_2(x;\lambda)$
 are linearly independent. Therefore, for these values of $\lambda$, $\vek\omega_1(x;\lambda)$ and
 $\vek\omega_2(x;\lambda)$ form a fundamental system of solutions of the system \eqref{3}.
 Thus, $\vec f(x)$ is orthogonal to all solutions of the system \eqref{3}.
 Consequently, $\vec f(x)\equiv 0$ and this completes the proof.

 \end{proof}


 \section{Riesz basis property of the SEAF}

 In this section some sufficient conditions for the Riesz basis
 property of the SEAF of the system \eqref{3} with separated
 $\lambda$-depending boundary conditions will be established.

First recall the definition of the Riesz basis.

 \begin{definition}

 A system of vectors $\{ \psi_n \}_{n=1}^\infty$ is called a Riesz basis in the Hilbert space $H$
 if there exists a bounded operator $A$ with bounded inverse $A^{-1},$ such that the transformed
  system $\{ A\psi_n \}_{n=1}^\infty$ forms an orthonormal basis in $H$.

 \end{definition}

 The following lemma is well known (see \cite{GohKr}).

 \begin{lemma}\label{lem2.1}

 Let the system of the vectors $\{ \psi_n \}_{n=1}^\infty$ be complete in a Hilbert space $H$.
 Let $\{ \varphi_n \}_{n=1}^\infty$ be a Riesz basis of $H$
 such that $\sum_{n=1}^\infty \|\psi_n-\phi_n\|^2<\infty$.
 Then the system $\{ \psi_n \}_{n=1}^\infty$ is a Riesz basis of $H$, too.

 \end{lemma}

 Also, the following lemma will be needed, which concerns the spectrum of the system \eqref{3}
 with separated $\lambda$-depending boundary conditions.

 \begin{lemma} \label {lem2.2}

 Let the function Q(x) be differentiable.

 Let $P_{11}(\lam)$ and $P_{12}(\lam)$ be relatively prime polynomials
 with $\deg P_{11}=\deg P_{12}=N_0$ and let $P_{21}(\lam)$ and
 $P_{22}(\lam)$  be relatively prime polynomials with $\deg P_{21}=\deg P_{22}=N_1$.

 Let $C_{ij}$ be the leading coefficient of the polynomial $P_{ij}(\lam)$
 and denote $C_{1}=C_{11}C_{21}$ and $C_{2}=C_{12}C_{22}$.

 Let the set $\Lambda$ contain $N=N_0+N_1$ eigenvalues of the
 problem \eqref{3} with separated $\lambda$-depending boundary
 conditions given by
 \begin{equation}  \label {b1}
 \begin{cases}
 P_{11}(\lam)y_1(0)+P_{12}(\lam)y_2(0)=0 \\
 P_{21}(\lam)y_1(1)+P_{22}(\lam)y_2(1)=0
 \end{cases}.
 \end{equation}

 Then it is possible to enumerate the remaining eigenvalues, such that
 \begin{equation} \label{l22.0}
 \lambda_n=\frac{i\ln (C_1/C_2)+2\pi n}{b-a}+O\left(\frac 1{|n|}\right),
 \quad \text{ where } n\in\dZ.
 \end{equation}
 \end{lemma}

 \begin{proof}
 The characteristic function $\chi(\lambda)$ of the system \eqref{3}
 with the boundary conditions in \eqref{b1} has the form
 \begin{equation*}
 \chi(\lambda)=
 P_{11}(\lambda)(P_{21}(\lambda)\psi_{01}(1;\lambda)+P_{22}(\lambda)\psi_{02}(1;\lambda))-
 P_{12}(\lambda)(P_{21}(\lambda)\f_{01}(1;\lambda)+P_{22}(\lambda)\f_{02}(1;\lambda)),
 \end{equation*}
 where the products of polynomials are all of degree $N$ by assumptions.

 Introduce the function
 \begin{equation*}
 \Pi(\lam)=\prod_{\tilde\lam_s\in\Lambda}(\lam-\tilde\lam_s)^{p_s},
 \end{equation*}
 where $p_s$ is the multiplicity of the eigenvalue $\tilde\lam_s$ in
 the set $\Lambda$.

 Then the eigenvalues which do not belong to the set $\Lambda$,
 are the roots the entire function
 \begin{equation*}
 \tilde\chi(\lambda):=\frac{\chi(\lambda)}{\Pi(\lam)}.
 \end{equation*}
 It follows from Lemma \ref{lem.2} that the function $\tilde\chi(\lambda)$
 satisfies the following estimate:
 \begin{equation}\label{l22.1}
 \tilde\chi(\lambda)=C_1\exp(a\lambda i)-C_2\exp(b\lambda i)+
 \frac1\lambda O(\max\{\exp(a\lambda i),\exp(b\lambda i)\}),
 \end{equation}
 where $C_1C_2\ne 0.$

 On the line
 \begin{equation*}
 \Re\lambda=\frac{\Im\ln(C_1/C_2)+(2n+1)\pi}{b-a},
 \end{equation*}
 which is determined by the equation $\arg(C_1\exp(a\lambda i))=\arg(-C_2\exp(b\lambda i))$,
 one has
 \begin{equation}\label{l22.2}
 |C_1\exp(a\lambda i)-C_2\exp(b\lambda i)|=|C_1\exp(a\lambda i)|+|C_2\exp(b\lambda i)|.
 \end{equation}

 From \eqref{l22.1} and \eqref{l22.2} one concludes that on this
 line, with $|\lambda|$ large enough,
 \begin{equation}\label{l3.3}
 |\tilde\chi(\lambda)-(C_1\exp(a\lambda i)-C_2\exp(b\lambda i))|<|C_1\exp(a\lambda i)-C_2\exp(b\lambda i)|.
 \end{equation}

 Therefore, it follows from Rouche's theorem (see \cite{Rudin}) and the estimates \eqref{l22.1}
 and \eqref{l3.3} that for $|n |$ large enough there exists precisely one root of $\chi(\lambda)$ in the strip
 \begin{equation}\label{l3.4}
 (2n-1)\pi<(b-a)\Re\lambda+\Im\ln(C_1/C_2)<(2n+1)\pi,
 \end{equation}
 and, furthermore, that there are $2|n|-1$ roots of the function $\tilde\chi(\lambda)$ in the strip
 \begin{equation}\label{l3.5}
 -(2|n|-1)\pi<(b-a)\Re\lambda+\Im\ln(C_1/C_2)<(2|n|-1)\pi.
 \end{equation}

 Therefore the roots of $\tilde\chi(\lambda)$ except for, possibly, a finite
 number of them, are simple. Moreover, the roots $\lambda_n$ of
 $\tilde\chi (\lambda)$ can be ordered as a bilateral sequence,
 so that for  $|n|>n_0$, \ $\lambda_n$  belongs to the strip (\ref {l3.4}).

 Let
 \begin{equation}\label{l3.6}
 \lambda_{n,0}=\frac{i\ln (C_1/C_2)+2\pi n}{b-a}
 \end{equation}
 be the root of the function
 \begin{equation*}
 \tilde\chi_0(\lambda)=C_1\exp(a\lambda i)-C_2\exp(b\lambda i).
 \end{equation*}
 Consider a disk $D(n,\rho)$ with the radius $\rho$ and the center $\lambda_{n,0}$.
 For $\lambda\in D(n,\rho)$ it follows from \eqref{l22.1} that there exist $K_1$, such that
 \begin{equation}\label{l3.7}
 |\lambda||\tilde\chi(\lambda)-\tilde\chi_0(\lambda)|<K_2(|\exp(a\Im\lambda)|+|\exp(b\Im\lambda)|)<K_1.
 \end{equation}
 Moreover, because $\Im\lambda$ is independent of $n$, $K_1$ is independent of $n,$ too.

 On the other hand, because the derivative of the function $\tilde\chi_0(\lambda)$ at $\lambda_{n,0}$ is nonzero, then,
 for $\rho$ small enough there exists $K_3$, such that $|\tilde\chi_0(\lambda)|>K_3 |\lambda-\lambda_0|$.

 Therefore, if $\rho\ge\frac{K_1}{K_3|\lambda|}$ and $|\lambda-\lambda_0|=\rho$, then \eqref{l3.7} implies that
 \begin{equation}\label{l3.8}
 |\tilde\chi_0(\lambda)|>\frac{K_1}{|\lambda|}>K_2\frac{|\exp(a\Im\lambda)|+|\exp(b\Im\lambda)|}{|\lambda|}>
 |\tilde\chi(\lambda)-\tilde\chi_0(\lambda)|.
 \end{equation}

 Hence, again by Rouche's theorem, in this disk the functions $\tilde\chi(\lambda)$ and $\tilde\chi_0(\lambda)$
 have the same number of roots, i.e., precisely one root.

 Therefore, $|\lambda_n-\lambda_{n,0}|<\frac{K_1}{K_3|\lambda_n|}$. Since
 $\lambda_n=\frac{2\pi n}{b-a}+O(1)$, one has $\frac{K_1}{K_3|\lambda_n|}=O(\frac1{|n|})$.

 Now, using the formula \eqref{l3.6}, the statement in \eqref{l22.0} follows.
 \end{proof}


 Lemmas \ref{lem2.1} and \ref{lem2.2} are used to prove the following theorem.

 \begin{theorem} \label {th2.2}

 Let $P_{11}(\lam)$ and $P_{12}(\lam)$ be relatively prime polynomials
 with $\deg P_{11}=\deg P_{12} = N_0$ and let $P_{21}(\lam)$ and
 $P_{22}(\lam)$  be relatively prime polynomials with $\deg P_{21}=\deg P_{22}=N_1$.

 Let $\Phi$ be a set, which consists of $N=N_0+N_1$ eigenfunctions and associate
 functions of the problem \eqref{3}, \eqref{b1} and assume that the SEAF of this problem
 without the set $\Phi$ is complete in the space $L^2[0,1]\oplus L^2[0,1]$.

 Then the SEAF of problem \eqref{3}, \eqref{b1} without the set $\Phi$
 is a Riesz basis in the space $L^2[0,1]\oplus L^2[0,1]$.

 \end{theorem}

 \begin{proof}
 By Lemma \ref{lem2.2}, it is possible to enumerate the eigenvalues $\lambda_n$,
 corresponding to the eigenfunctions $\vek{\omega_n}(x)$ which are not contained in set $\Phi$,
 such that
 \begin{equation*}
 \lambda_n=\frac{i\ln (C_1/C_2)+2\pi n}{b-a}+O\left(\frac 1{|n|}\right).
 \end{equation*}

 Because $\vek{\omega_n}(x)$ satisfies the first of the conditions in \eqref{b1}, it may be written in the form
 $\vek{\omega_n}(x)=P_{12}(\lambda_n)\vek{\f_0}(x;\lambda_n)-
 P_{11}(\lambda_n)\vek{\psi_0}(x;\lambda_n)$ (up to a constant multiplier).

 Then, by Lemma \ref{lem.2},
 \begin{equation} \label {t22.2}
 \vek{\omega_n}(x)=
 \binom{C_{12}\exp(a\lambda_n ix)}{-C_{11}\exp(b\lambda_n ix)}+
 \frac 1{\lambda_n} (O(\exp(a\lambda_n ix))+O(\exp(b\lambda_n ix))).
 \end{equation}

 Combining the estimates \eqref{l22.0} in Lemma \ref{lem2.2} with \eqref{t22.2} one
 obtains
 \begin{multline} \label {t22.3}
 \vek{\omega_n}(x)=
 \binom{C_{12}\exp(a\frac{i\ln (C_1/C_2)+2\pi n}{b-a} ix)}
 {-C_{11}\exp(b\frac{i\ln (C_1/C_2)+2\pi n}{b-a} ix)}+
 \frac 1n (O(\exp(a\frac{i\ln (C_1/C_2)+2\pi n}{b-a} ix)) \\
     + O(\exp(b\frac{i\ln (C_1/C_2)+2\pi n}{b-a} ix)))=
 \binom{C_{12}\exp(a\frac{i\ln (C_1/C_2)+2\pi n}{b-a} ix)}
 {-C_{11}\exp(b\frac{i\ln (C_1/C_2)+2\pi n}{b-a} ix)}+O(\frac 1n).
 \end{multline}

 Now define the operator $A:L_2[0,1]\oplus L_2[0,1]\to L_2[a,b]$ via
 \begin{equation} \label {t22.4}
 A\binom{y_1}{y_2}(x)=
 \begin{cases}
 \frac 1{C_{12}}y_1(\frac xa), \quad
 &\text{ where } a<x<0 \\
 -\frac 1{C_{11}}y_2(\frac xb), \quad &\text{ where } 0<x<b .
 \end{cases}
 \end{equation}
 Then $A$ and $A^{-1}$ are bounded. Therefore, the system $\vek{\omega_n}(x)$
 is a Riesz basis in the space $L_2[0,1]\oplus L_2[0,1]$ if and only if the system $A(\vek{\omega_n})$
 is a Riesz basis in the space $L_2[a,b]$.

 From the estimate \eqref{t22.3} and the definition of $A$ in \eqref{t22.4} one
 obtains
 \begin{equation} \label {t22.5}
 A(\vek{\omega_n})=\exp(\frac{i\ln (C_1/C_2)+2\pi n}{b-a} ix)+O(\frac 1n).
 \end{equation}
 It is obvious, that the system
 \begin{equation*}
 \tilde \omega_n=\exp(\frac{i\ln (C_1/C_2)+2\pi n}{b-a} ix)
 \end{equation*}
 is an orthogonal basis in the space $L_2[a,b]$ and that the norms of $\tilde \omega_n$ are given by
 \begin{equation*}
 \|\tilde \omega_n\|=\int_a^b\exp(-2\frac{\Re(\ln (C_1/C_2))}{b-a} x)\, dx
 \end{equation*}
 for all $n$.
 From the estimate \eqref{t22.5} one concludes that
 \begin{equation*}
 \sum_{n=-\infty}^{\infty}\|A(\vek{\omega_n})-\tilde \omega_n\|^2<\infty.
 \end{equation*}
 Therefore, by Lemma \ref{lem2.1}, $A(\vek{\omega_n})$ is a Riesz
 basis in $L_2[a,b]$.
 \end{proof}

 From Theorem \ref{th2.2} we obtained the following result.

 \begin{theorem} \label {th2.3}

 Let $\deg P_{11}=\deg P_{12}=0$, i.e., $P_{11}\ne 0$ and $P_{12}\ne 0$ are constants.
 Let $P_{21}(\lam)$ and $P_{22}(\lam)$  be relatively prime
 polynomials with $\deg P_{21}=\deg P_{22}=N$.

 Let $\Phi$ be a set, which consists of $N$ eigenfunctions and associate functions,
 which satisfies the following condition:

 If $\Phi$ contains either an eigenfunction or an associate function
 corresponding to an eigenvalue $\lam_k$, then it also contains
 all the associate functions of higher order corresponding to
 the same eigenvalue.

 Then the SEAF of the problem \eqref{3}, \eqref{b1} without the set $\Phi$
 is a Riesz basis in the space $L^2[0,1]\oplus L^2[0,1]$.
 \end{theorem}

 \begin{proof}
 In this case, by Theorem \ref{theo.1}, the SEAF of the problem \eqref{3}, \eqref{b1}
 without the set $\Phi$ is complete in the space $L^2[0,1]\oplus L^2[0,1]$.

 Therefore, by Theorem \ref{th2.2} it is a Riesz basis.
 \end{proof}

By taking $N=0$ in Theorem \ref{th2.3} yields the following
corollary for the system \eqref{3} with boundary conditions not
depending on a spectral parameter, see \cite{TroYa}.

 \begin{corollary} \label{c.2}
 (\cite{TroYa}) Let $h_1$ and $h_2$ be nonzero numbers.
 Then the SEAF of problem \eqref{3} with boundary conditions
 \begin{equation}\label{c2.1}
 \begin{cases}
 y_1(0)+h_1y_2(0)=0 \\
 y_1(1)+h_2y_2(1)=0,
 \end{cases}
 \end{equation}
 is a Riesz basis in the space $L^2[0,1]\oplus L^2[0,1]$.
 \end{corollary}


 \end{document}